\documentclass[12pt,leqno]{amsart}
\usepackage{amsmath,amssymb,amsfonts}
\usepackage{eucal,graphicx,colortbl,arydshln}

\newcommand \comment[1]{}			

\newtheorem{lem}{Lemma}
\newtheorem{cor}[lem]{Corollary}

\newtheorem{thm}[lem]{Theorem}

  \newtheorem{exama}[lem]{Example}
\newenvironment{exam}{\begin{exama}\rm}{\end{exama}}

\newcommand \myexam[1]{\smallskip\begin{example}[\emph{#1}]}

\renewcommand{\phi}{\varphi}
\renewcommand\emptyset{\varnothing}

\newcommand\lcmd{\operatorname{lcmd}}
\newcommand\lcm{\operatorname{lcm}}
\newcommand\LCM{\mathop{\operatorname{LCM}}}

\allowdisplaybreaks

\begin{document}

\title[Determinants in the Kronecker product of matrices]{Determinants in the Kronecker product of matrices:  The incidence matrix of a complete graph}

\author{Christopher R.H.\ Hanusa}
\address{Department of Mathematics\\
Queens College (CUNY)\\
65-30 Kissena Blvd.\ \\
Flushing, NY 11367, U.S.A.\ \\
phone: 718-997-5964}
\email{\tt chanusa@qc.cuny.edu}

\author{Thomas Zaslavsky}
\address{Department of Mathematical Sciences\\
Binghamton University (SUNY)\\
Binghamton, NY 13902-6000, U.S.A.\ }
\email{\tt zaslav@math.binghamton.edu}

\begin{abstract}
We investigate the least common multiple of all subdeterminants, 
$\lcmd(A\otimes B)$, of a Kronecker product of matrices, of which one is an integral matrix $A$ with two columns and the other is the incidence matrix of a complete graph with $n$ vertices.  We prove that this quantity is the least common multiple of $\lcmd(A)$ to the power $n-1$ and certain binomial functions of the entries of $A$. 
\end{abstract}



\maketitle
\pagestyle{headings}

\section{Introduction}

In a study of non-attacking placements of chess pieces, Chaiken, Hanusa, and Zaslavsky \cite{QQ} were led to a quasipolynomial formula that depends in part on the least common multiple of the determinants of all square submatrices of a certain Kronecker product matrix, namely, the Kronecker product of an integral $2 \times 2$ matrix $A$ with the incidence matrix of a complete graph.  We give a compact expression for the least common multiple of the subdeterminants of this product matrix, generalized to $A$ of order $m \times 2$.

\section{Background}

\subsection*{Kronecker product}
For matrices $A=(a_{ij})_{m\times k}$ and $B=(b_{ij})_{n\times l}$, the Kronecker product $A\otimes B$ is defined to be the $mn \times kl$ block matrix
\begin{equation*}
\left(\begin{array}{c:c:c} 
a_{11}B & \cdots & a_{1k}B \\ \hdashline 
\vdots & \ddots & \vdots \\ \hdashline 
a_{m1}B & \cdots & a_{mk}B \\
\end{array}\right).
\end{equation*}
It is known (see \cite{HJ}, for example) that when $A$ and $B$ are square matrices of orders $m$ and $n$, respectively, then $\det(A\otimes B)=\det(A)^n\det(B)^m$.  

\subsection*{The $\lcmd$ operation}
The quantity we want to compute is $\lcmd(A\otimes B)$, where for an integer matrix $M$, the notation $\lcmd(M)$ denotes the least common multiple of the determinants of all square submatrices of $M$.  This is a much stronger question, as the matrices $A$ and $B$ are most likely not square and the result depends on all square submatrices of their Kronecker product.  We discuss properties of this operation in Section~\ref{sec:properties}, after introducing our main result in Section~\ref{sec:main}.

\subsection*{Incidence matrix}
For a simple graph $G=(V,E)$, the incidence matrix $D(G)$ is a $|V|\times|E|$ matrix with a row corresponding to each vertex in $V$ and a column corresponding to each edge in $E$.  For a column that corresponds to an edge $e=vw$, there are exactly two non-zero entries: one $+1$ and one $-1$ in the rows corresponding to $v$ and $w$.  The sign assignment is arbitrary.  
The complete graph $K_n$ is the graph on $n$ vertices $v_1,\hdots,v_n$ with an edge between every pair of vertices.  Its incidence matrix has order $n \times \binom{n}{2}$.  

\medskip
Of interest in this article are Kronecker products of the form $A~\otimes~D(K_n)$.  

\begin{exam}\label{X:illustrative} 
We present an illustrative example that we will revisit in the proof of our main theorem.  We consider $K_4$ to have vertices $v_1$ through $v_4$, corresponding to rows $1$ through $4$ of $D(K_4)$, and edges $e_1$ through $e_6$, corresponding to columns $1$ through $6$ of $D(K_4)$. One of the many incidence matrices for $K_4$ is the $4\times 6$ matrix
\begin{equation*}
D(K_4)=\begin{pmatrix} 
1 & 1 & 1 & 0 & 0 & 0 \\
-1 & 0 & 0 & 1 & 1 & 0 \\
0 & -1 & 0 & -1 & 0 & 1 \\
0 & 0 & -1 & 0 & -1 & -1 \\
\end{pmatrix}.
\end{equation*}
If $A$ is the $3 \times 2$ matrix $\begin{pmatrix} a_{11} & a_{12} \\ a_{21} & a_{22} \\ a_{31} & a_{32}\\ \end{pmatrix}$, we investigate the Kronecker product 
\begin{gather*}
A \otimes D(K_4) = \\
\small
\left(\begin{array}{c@{\ }c@{\ }c@{\ }c@{\ }c@{\ }c:c@{\ }c@{\ }c@{\ }c@{\ }c@{\ }c} 
a_{11} & a_{11} & a_{11} & 0 & 0 & 0 & a_{12} & a_{12} & a_{12} & 0 & 0 & 0 \\
-a_{11} & 0 & 0 & a_{11} & a_{11} & 0 & -a_{12} & 0 & 0 & a_{12} & a_{12} & 0 \\
0 & -a_{11} & 0 & -a_{11} & 0 & a_{11} & 0 & -a_{12} & 0 & -a_{12} & 0 & a_{12} \\
0 & 0 & -a_{11} & 0 & -a_{11} & -a_{11} & 0 & 0 & -a_{12} & 0 & -a_{12} & -a_{12} \\\hdashline 
 a_{21} &  a_{21} &  a_{21} &  0 &  0 &  0 &  a_{22} &  a_{22} &  a_{22} &  0 &  0 &  0 \\
-a_{21} &  0 &  0 &  a_{21} &  a_{21} &  0 & -a_{22} &  0 &  0 &  a_{22} &  a_{22} &  0 \\
 0 & -a_{21} &  0 & -a_{21} &  0 &  a_{21} &  0 & -a_{22} &  0 & -a_{22} &  0 &  a_{22} \\
 0 &  0 & -a_{21} &  0 & -a_{21} & -a_{21} &  0 &  0 & -a_{22} &  0 & -a_{22} & -a_{22} \\\hdashline 
 a_{31} &  a_{31} &  a_{31} &  0 &  0 &  0 &  a_{32} &  a_{32} &  a_{32} &  0 &  0 &  0 \\
-a_{31} &  0 &  0 &  a_{31} &  a_{31} &  0 & -a_{32} &  0 &  0 &  a_{32} &  a_{32} &  0 \\
 0 & -a_{31} &  0 & -a_{31} &  0 &  a_{31} &  0 & -a_{32} &  0 & -a_{32} &  0 &  a_{32} \\
 0 &  0 & -a_{31} &  0 & -a_{31} & -a_{31} &  0 &  0 & -a_{32} &  0 & -a_{32} & -a_{32} \\
\end{array}\right).
\normalsize
\end{gather*}
\end{exam}

\subsection*{Submatrix notation}
Let $A=(a_{ij})$ be an $m\times 2$ matrix; this makes $A\otimes D(K_n)$ an $mn\times n(n-1)$ matrix with non-zero entries $\pm a_{ij}$.  
We introduce new notation for some matrices that will arise naturally in our theorem.  For $i,j\in[m]:=\{1,2,\ldots,m\}$, we write $A^{i,j}$ to represent $\begin{pmatrix} a_{i1} & a_{i2} \\ a_{j1} & a_{j2} \end{pmatrix}$.  If $I$ is a multisubset of $[m]$, we define $a_{Ik}$ to be the product $\prod_{i\in I} a_{ik}$.  If $I$ and $J$ are multisubsets of $[m]$, we define $A^{I,J}$ to be the matrix $\begin{pmatrix} a_{I1} & a_{I2} \\ a_{J1} & a_{J2} \end{pmatrix}.$ In this notation, 
$$
\lcmd A = \lcm\big( \LCM_{i,k} a_{ik}, \LCM_{i,j} \det A^{i,j} \big),
$$
where $\LCM$ denotes the least common multiple of non-zero quantities taken over all indicated pairs of indices.

\section{Main Theorem and Main Corollary}
\label{sec:main}

Let 
\begin{align*}
\mathcal{K}_m := \{ (I,J) : \ & I, J \text{ are multisubsets of } [m] \ \\& \text{such that } |I|=|J| \text{ and } I \cap J = \emptyset \}.
\end{align*}
Recall that a \emph{subdeterminant} or \emph{minor} of a matrix is the determinant of a square submatrix.

\begin{thm}
Let $A$ be an $m\times 2$ matrix, not identically zero, and $n \geq 1$.  The least common multiple of all subdeterminants of $A\otimes D(K_n)$ is
\begin{equation}
\begin{aligned}
&\lcmd\big(A\otimes D(K_n)\big) \\
&\qquad = \lcm\bigg(
(\lcmd A)^{n-1}, \LCM_{\mathcal{K}} \, \Big[\prod_{(I_s,J_s)\in\mathcal{K}}\det A^{I_s,J_s}\Big]\bigg),
\label{eq:lcmr}
\end{aligned}
\end{equation}
where $\LCM_\mathcal{K}$ denotes the least common multiple of non-zero quantities taken over all collections $\mathcal{K} \subseteq \mathcal{K}_m$ such that $2 \sum_{(I,J)\in\mathcal{K}} |I| \leq n$.
\label{thm:lcmr}
\end{thm}

The proof, which is long, is in Section \ref{proof} at the end of this article. Although the expression is not as simple as we wanted, we were fortunate to find it; it seems to be a much harder problem to get a similar formula when $A$ has more than two columns.

Note that it is only necessary to take the $\LCM$ component over all maximal collections $\mathcal{K}$, that is, collections $\mathcal{K}$ satisfying $\sum|I_s|=\lfloor n/2 \rfloor$.

When understanding the right-hand side of Equation \eqref{eq:lcmr}, it may be instructive to notice that the $\LCM$ factor on the right-hand side divides 
\begin{equation*}
\prod_{\substack{\textup{disjoint $I,J$:}\\|I|=|J|=p}} \, (\det A^{I,J})^{\lfloor n/2p \rfloor},
\end{equation*}
since the largest number of individual $\det A^{I,J}$ factors that may occur for disjoint $p$-member multisubsets $I$ and $J$ of $[m]$ is $\lfloor n/2p \rfloor$.

\medskip
When $m=2$, the only pair of disjoint $p$-member multisubsets of $[m]$ is $\{1^p\}$ and $\{2^p\}$.  From this, we have the following corollary.

\begin{cor}
Let $A$ be a $2\times 2$ matrix, not identically zero, and $n \geq 1$.  The least common multiple of all subdeterminants of $A\otimes D(K_n)$ is
\begin{align*}
&\lcmd\big(A\otimes D(K_n)\big) \\
&\qquad = \lcm\big((\lcmd A)^{n-1}, \LCM_{p=2}^{\lfloor n/2 \rfloor} \big((a_{11}a_{22})^p-(a_{12}a_{21})^p \big)^{\lfloor n/2p \rfloor}\big),
\label{eq:lcm2}
\end{align*}
where $\LCM$ denotes the least common multiple over the range of $p$.
\label{thm:lcm2}
\end{cor}

\section{Properties of the $\lcmd$ Operation}
\label{sec:properties}

Four kinds of operation on $A$ do not affect the value of $\lcmd A$:  permuting rows or columns, duplicating rows or columns, adjoining rows or columns of an identity matrix, and transposition.  The first two will not change the value of $\lcmd (A \otimes D(K_n))$.  However, the latter two may.  
According to Corollary \ref{thm:lcm2}, transposing a $2\times2$ matrix $A$ does not alter $\lcmd\big(A\otimes D(K_n)\big)$; but when $m>2$ that is no longer the case, as Example 2 shows.  Adding columns of an identity matrix also may change the l.c.m.d., even when $A$ is $2\times2$; also see Example 2.  However, we may freely adjoin rows of $I_2$ if $A$ has two columns.

\begin{cor}
Let $A$ be an $m \times 2$ matrix, not identically zero, and $n \geq 1$.  Let $A'$ be $A$ with any rows of the $2\times2$ identity matrix adjoined.  Then 
$$
\lcmd\big(A'\otimes D(K_n)\big) = \lcmd\big(A\otimes D(K_n)\big).
$$
\label{C:addIrows}
\end{cor}

\begin{proof}
It suffices to consider the case where $A'$ is $A$ adjoin an $(m+1)^\textup{st}$ row $\begin{pmatrix} 1 & 0 \end{pmatrix}$. It is obvious that $\lcmd A' = \lcmd A$; this accounts for the first component of the least common multiple in Equation \eqref{eq:lcmr}.  

For the second component, any $\mathcal{K}$ that appears in the $\LCM$ for $A$ also appears for $A'$. 
Suppose $\mathcal{K'}$ is a collection that appears only for $A'$; this implies that in $\mathcal{K'}$ there exist pairs $(I_s,J_s)$ such that $m+1 \in I_s$ (or $J_s$, but that case is similar).  
Since $a_{I_s2} = 0$, $\det A^{I_s,J_s} = a_{I_s1} a_{J_t2}$, which is a product of at most $n-1$ elements of $A$.  This, in turn, divides $(\lcmd A)^{n-1}$.  We conclude that the right-hand side of Equation \eqref{eq:lcmr} is the same for $A'$ as for $A$.
\end{proof}

We do not know whether or not adjoining a row of the identity matrix to an $m \times l$ matrix $A$ preserves $\lcmd\big(A\otimes D(K_n)\big)$ when $l>2$.  Limited calculations give the impression that this may indeed be true.

\section{Examples}

We calculate a few examples with matrices $A$ that are related to those needed for the chess-piece problem of \cite{QQ}.  In that kind of problem the matrix of interest is $M \otimes D(K_n)^T$ where $M$ is an $m \times 2$ matrix.  Hence, in Theorem \ref{thm:lcmr} we want $A = M^T$, so Theorem \ref{thm:lcmr} applies only when $m \leq 2$.
\medskip    

\noindent
{\bf Example 1.}  
When the chess piece is the bishop, $M$ is the $2\times 2$ symmetric matrix 
$$
M_B = \begin{pmatrix} 1 & 1 \\ 1 & -1 \end{pmatrix}.
$$
We apply Corollary \ref{thm:lcm2} with $A=M_B$, noting that $\lcmd(A)=2$.  We get
$$
\lcmd\big(A\otimes D(K_n)\big) = \lcm\big(2^{n-1}, \LCM_{p=2}^{\lfloor n/2 \rfloor} \big((-1)^p-1^p \big)^{\lfloor n/2p \rfloor}\big).
$$
The $\LCM$ generates powers of $2$ no larger than $2^{n/2}$, hence $\lcmd\big(M_B\otimes D(K_n)\big) = 2^{n-1}$.

\medskip
\noindent
{\bf Example 2.}  
When the chess piece is the queen, $M$ is the $4\times 2$ matrix 
$
M_Q = \begin{pmatrix} I \\ M_B \end{pmatrix}
$
with $\lcmd(A)=2$.  Then our matrix 
$A = M_Q^T = \begin{pmatrix} I & M_B \end{pmatrix}$.  
Since $M_Q^T$ has four columns Theorem \ref{thm:lcmr} does not apply.  In fact, we found that $\lcmd\big(M_Q^T\otimes D(K_4)\big) = 24$, quite different from $\lcmd\big(M_B \otimes D(K_4)\big) = 8.$

However, if we take $A = M_Q$ instead of $M_Q^T$, Corollary \ref{C:addIrows} applies; we conclude that $\lcmd\big(M_Q \otimes D(K_n)\big) = \lcmd\big(M_B\otimes D(K_n)\big) = 2^{n-1}$.  

Thus, $A = M_Q$ is an example where transposing $A$ changes the value of $\lcmd\big(A \otimes D(K_n) \big)$ dramatically.

\medskip
\noindent
{\bf Example 3.}  
A more difficult example is the fairy chess piece known as a nightrider, which moves an unlimited distance in the directions of a knight.  Here $M$ is the $4 \times 2$ matrix 
$$
M_N = \begin{pmatrix} 1 & 2 \\ 2 & 1 \\ 1 & -2 \\ 2 & -1 \end{pmatrix}.
$$
We can use Theorem \ref{thm:lcmr} to calculate $\lcmd\big(M_N \otimes D(K_n)\big)$.  
Since what is needed for the chess problem is $\lcmd\big(M_N^T \otimes D(K_n)\big)$, this example does not help in \cite{QQ}; nevertheless it makes an interestingly complicated application of Theorem \ref{thm:lcmr}.

The submatrices
$$
\begin{pmatrix} 1 & 2 \\ 2 & 1 \end{pmatrix},
\begin{pmatrix} 1 & 2 \\ 1 & -2 \end{pmatrix}, \textup{ and }
\begin{pmatrix} 1 & 2 \\ 2 & -1 \end{pmatrix},
$$
with determinants $-3$, $-4$, and $-5$, respectively, lead to the conclusion that $\lcmd(A)=60$.  Every pair $(I,J)$ of disjoint $p$-member multisubsets of $[4]$ has one of the following seven forms, up to the order of $I$ and $J$:  
\begin{gather*}
\big( \{1^q\}, \{2^r,3^s,4^t\} \big), \quad 
\big( \{2^r\}, \{1^q,3^s,4^t\} \big),\\
\big( \{3^s\}, \{1^q,2^r,4^t\} \big),\quad
\big( \{4^t\}, \{1^q,2^r,3^s\} \big),\\
\big( \{1^q,2^r\}, \{3^s,4^t\} \big), \quad
\big( \{1^q,3^s\}, \{2^r,4^t\} \big), \quad
\big( \{1^q,4^t\}, \{2^r,3^s\} \big),
\end{gather*}
where the sum of the exponents in each multisubset is $p$, and where $q$, $r$, $s$, and $t$ may be zero.  
It turns out that $\det A^{I,J}$ has the same form in all seven cases: precisely $\pm 2^{u}(2^{2p-2u}\pm 1)$, where $u$ is a number between $0$ and $p$.  
Furthermore, every value of $u$ from $0$ to $p$ appears and every choice of plus or minus sign appears (except when $u = p$) in $\det A^{I,J}$ for some choice of $(I,J)$.  We present two representative examples that support this assertion.

The case of $I=\{1^q\}$ and $J=\{2^r,3^s,4^t\}$.  Then 
$$A^{I,J}=\begin{pmatrix} 1^q & 2^q \\ 2^r1^s2^t & 1^r(-2)^s(-1)^t \end{pmatrix},$$
with $q=r+s+t=p$.  We can rewrite $\det A^{I,J}$ as ${}\pm 2^s - 2^{2p-s}=-2^s(2^{2p-2s}\pm1)$.  The only instance in where there is no choice of sign is when $s=p$ and $r=t=0$, in which case $\det A^{I,J}$ simplifies to either $0$ or $-2^{p+1}$.

The case of $I=\{1^q,2^r\}$ and $J=\{3^s,4^t\}$.  Then 
$$A^{I,J}=\begin{pmatrix} 1^q2^r & 2^q1^r \\ 1^s2^t & (-2)^s(-1)^t \end{pmatrix},$$
where $q+r=s+t=p$.  For this choice of $I$ and $J$, $\det A^{I,J}=(-1)^p2^{r+s}-2^{2p-r-s}$.

Since every $\det A^{I,J}$ has the same form, and at most $\lfloor p/2n\rfloor$ factors of type $(2^{2p-2u}\pm 1)$ may occur at the same time, the $\LCM$ in Equation \eqref{eq:lcmr} is exactly
$$
\LCM_{\mathcal{K}} \, \bigg(\prod_{(I_s,J_s)\in\mathcal{K}}\det A^{I_s,J_s}\bigg) = 2^{N}\LCM_{\substack{1\leq p\leq n/2 \\ 0\leq u\leq p-1}} (2^{2p-2u} \pm 1)^{\lfloor n/2p \rfloor},
$$
for some $N\leq n$.  We conclude that
$$
\lcmd\big(A\otimes D(K_n)\big) = \lcm(60^{n-1},\LCM_{\substack{1\leq p\leq n/2 \\ 0\leq u\leq p-1}} (2^{2p-2u} \pm 1)^{\lfloor n/2p \rfloor}).
$$
As a sample of the type of answer we get, when $n=8$ this expression is 
\begin{align*}
\lcmd&\big(A\otimes D(K_8)\big) \\
& =  \lcm(60^{7}, (4\pm 1)^{\lfloor 8/2\rfloor}, (16\pm 1)^{\lfloor 8/4\rfloor}, (64\pm 1)^{\lfloor 8/6\rfloor},(256\pm 1)^{\lfloor 8/8\rfloor}) \\
& =  60^7\cdot 7 \cdot 13 \cdot 17^2\cdot 257.
\end{align*}
The first few values of $n$ give the following numbers:
\medskip

\noindent
\begin{tabular}{@{}r|l|l@{}}
$n$ & $\lcmd\big(A\otimes D(K_8)\big)$ & (factored) \\ \hline
$2$ & 60 & $60^1$ \\
$3$ & 3600 & $60^2$ \\
$4$ & 3672000 & $60^3\cdot 17$ \\
$5$ & 220320000 & $60^4 \cdot 17$ \\
$6$ & 1202947200000 & $60^5 \cdot 7 \cdot 13 \cdot 17$ \\
$7$ & 72176832000000 & $60^6 \cdot 7 \cdot 13 \cdot 17$ \\
$8$ & 18920434740480000000 & $60^7\cdot 7 \cdot 13 \cdot 17^2\cdot 257$ \\
$9$ & 1135226084428800000000 & $60^8 \cdot 7 \cdot 13 \cdot 17^2 \cdot 257$ \\
$10$ & 952295753183943168000000000 & $60^9 \cdot 7 \cdot 11 \cdot 13 \cdot 17^2 \cdot 31 \cdot 41 \cdot 257$ \\
\end{tabular}

\section{Remarks}\label{remarks}

We hope to determine in the future whether $\lcmd(A\otimes B)$ has a simple form for arbitrary matrices $A$ and $B$.  Our limited experimental data suggests this may be difficult.  However, we think at least some generalization of Theorem \ref{thm:lcmr} is possible.  

We would like to understand, at minimum, why the theorem as stated fails when $B = D(K_n)$ and $A$ has more than two columns.

Another direction worth investigating is the number-theoretic aspects of Theorem \ref{thm:lcmr}.  

\section{Proof of the Main Theorem}\label{proof}

During the proof we refer from time to time to Example \ref{X:illustrative}, which will give a concrete illustration of the many steps.  We assume $a_{11}$, $a_{12}$, $a_{21}$, $a_{22}$, $a_{31}$, and $a_{32}$ are non-zero constants.  

\subsection{Calculating the determinant of a submatrix}
Consider an $l\times l$ submatrix $N$ of $A\otimes D(K_n)$.  We wish to evaluate the determinant of $N$ and show that it divides the right-hand side of Equation \eqref{eq:lcmr}.  We need consider only matrices $N$ whose determinant is not zero, since a matrix with $\det N = 0$ has no effect on the least common multiple.

Since $D(K_n)$ is constructed from a graph, we will analyze $N$ from a graphic perspective.  The matrix $N$ is a choice of $l$ rows and $l$ columns from $A\otimes D(K_n)$.  This corresponds to a choice of $l$ vertices and $l$ edges from $K_n$ where we are allowed to choose up to $m$ copies of each vertex and up to two copies of an edge.  Another way to say this is that we are choosing $m$ subsets of $V(K_n)$, say $V_1$ through $V_m$, and two subsets of $E(K_n)$, say $E_1$ and $E_2$, with the property that $\sum_{i=1}^m |V_i|=\sum_{k=1}^2 |E_k|=l$.  From this point of view, if a row in $N$ is taken from the first $n$ rows of $A\otimes D(K_n)$, we are placing the corresponding vertex of $V(K_n)$ in $V_1$, and so on, up through a row from the last $n$ rows of $A\otimes D(K_n)$, which corresponds to a vertex in $V_m$.  We will say that the copy of $v$ in $V_i$ is the {\em $i^\textup{th}$ copy of $v$} and the copy of $e$ in $E_k$ is the {\em $k^\textup{th}$ copy of $e$}.  

The order of $N$ satisfies $l \leq 2n-2$ because, if $N$ had $2n-1$ columns of $A\otimes D(K_n)$, then at least one edge set, $E_1$ or $E_2$, would contain $n$ edges from $K_n$.  The columns corresponding to these edges would form a dependent set of columns in $N$, making $\det N=0$.  

\medskip
In our illustrative example, choose the submatrix $N$ consisting of rows $1$, $5$, $7$, and $10$ and columns $1$, $4$, $7$, and $8$.  Then $N$ is the matrix
\begin{equation*}
N=\begin{pmatrix} 
 a_{11} &  0 &  a_{12} &  a_{12} \\
 a_{21} &  0 &  a_{22} &  a_{22} \\
 0 & -a_{21} &  0 & -a_{22} \\
-a_{31} &  a_{31} & -a_{32} &  0 \\
\end{pmatrix},
\end{equation*}
and in the notation above, $V_1=\{v_1\}$, $V_2=\{v_1,v_3\}$, $V_3=\{v_2\}$, $E_1=\{e_1,e_4\}$, and $E_2=\{e_1,e_2\}$.

\medskip
Returning to the proof, within this framework we will now perform elementary matrix operations on $N$ in order to make its determinant easier to calculate.  We call the resulting matrix the {\em simplified matrix of $N$}.  Each copy of a vertex $v$ has a row in $N$ associated with it; two rows corresponding to two copies of the same vertex contain the same entries except for the different multipliers $a_{ik}$.  For example, if $v$ is a vertex in both $V_1$ and $V_2$, then there is a row corresponding to the first copy with multipliers $a_{11}$ and $a_{12}$ and a row corresponding to the second copy with the same entries multiplied by $a_{21}$ and $a_{22}$.  

There cannot be a vertex in three or more vertex sets since then the corresponding rows of $N$ would be linearly dependent and $\det N$ would be zero.  

When there is a vertex in exactly two vertex sets $V_i$ and $V_j$ corresponding to two rows $R_i$ and $R_j$ in $N$, we perform the following operations depending on the multipliers $a_{i1}$, $a_{i2}$, $a_{j1}$, and $a_{j2}$.  We first notice that $\det A^{i,j}=a_{i1}a_{j2}-a_{i2}a_{j1}$ is non-zero; otherwise, the rows $R_i$ and $R_j$ would be linearly dependent in $N$ and $\det N=0$.  Therefore either both $a_{i1}$ and $a_{j2}$ or both $a_{i2}$ and $a_{j1}$ are non-zero.  In the former case, let us add $-a_{j1}/a_{i1}$ times $R_i$ to $R_j$ in order to zero out the entries corresponding to edges in $E_1$.  The multipliers of entries in $R_j$ corresponding to edges in $E_2$ are now all $\det A^{i,j}/a_{i1}$.  Similarly, we can zero out the entries in $R_i$ corresponding to edges in $E_2$.  Lastly, factor out $\det A^{i,j}/a_{j2}a_{i1}$ from $R_j$.  If on the other hand, either multiplier $a_{i1}$ or $a_{j2}$ is zero, then reverse the roles of $i$ and $j$ in the preceding argument.  These manipulations ensure that the multiplier of every non-zero entry in $N$ that corresponds to an $i^\textup{th}$ vertex and a $k^\textup{th}$ edge is $a_{ik}$. 

The appearance of a denominator, $a_{i1}a_{j2}$, in the factor $\det A^{i,j}/a_{j2}a_{i1}$ is merely an artifact of the construction; we could have cancelled it by factoring out $a_{i1}$ in row $i$ and $a_{j2}$ in row $j$.  However, if we had done this, the entries of the matrix would no longer be of the form $a_{ik}$, $-a_{ik}$, and $0$, which would make the record-keeping in the coming arguments more tedious.

\medskip
In our illustrative example, because $v_1$ is a member of both $V_1$ and $V_2$, we perform row operations on the rows of $N$ corresponding to $v_1$ to yield the simplified matrix of $N$:
\begin{equation*}
N_\textup{simplified}=\begin{pmatrix} 
 a_{11} &  0 &  0 &  0 \\
 0 &  0 &  a_{22} &  a_{22} \\
 0 & -a_{21} &  0 & -a_{22} \\
-a_{31} &  a_{31} & -a_{32} &  0 \\
\end{pmatrix}.
\end{equation*}
The determinants of $N$ and $N_\textup{simplified}$ are related by
$$\det N = \frac{\det A^{1,2}}{a_{11}a_{22}}\det N_\textup{simplified}.$$
The denominator $a_{11}a_{22}$ would disappear if we had chosen to factor out the $a_{11}$ in the first row and the $a_{22}$ in the second row of $N_\textup{simplified}$.

\medskip
Returning to the proof, we assert that the simplified matrix of $N$ has no more that two non-zero entries in any column.  For a column $e$ corresponding to an edge $e=vw$ in $K_n$, each of $v$ and $w$ is either in one vertex set $V_i$ or in two vertex sets $V_i$ and $V_j$.  If the vertex corresponds to two rows in $N$, the above manipulations ensure that there is only one copy of the vertex that has a non-zero multiplier in the column.  Another important quality of this simplification is that if a vertex is in more than one vertex set, then every edge incident with one instance of this repeated vertex is now in the same edge set; more precisely, if $v \in V_i \cap V_j$, then every edge incident with the $i^\textup{th}$ copy of $v$ is in $E_1$ and every edge incident with the $j^\textup{th}$ copy is in $E_2$, or vice versa.

Since we are assuming $\det N \neq 0$, $N$ has at least one non-zero entry in each column or row.  If a row (or column) has exactly one non-zero entry, we can reduce the determinant by expanding in that row (or column).  This contributes that non-zero entry as a factor in the determinant.  After reducing repeatedly in this way, we arrive at a matrix where each column has exactly two non-zero entries, and each row has at least two non-zero entries.  This implies that every row has exactly two non-zero entries as well.  After interchanging the necessary columns and rows and possibly multiplying columns by $-1$, the structure of what we will call the {\em reduced matrix of $N$} is a block diagonal matrix where each block $B$ is a weighted incidence matrix of a cycle, such as
\begin{equation*}
\begin{pmatrix} 
y_1 & 0 & 0 & 0 & 0 & -z_6 \\
-z_1 & y_2 & 0 & 0 & 0 & 0 \\
0 & -z_2 & y_3 & 0 & 0 & 0 \\
0 & 0 & -z_3 & y_4 & 0 & 0 \\
0 & 0 & 0 & -z_4 & y_5 & 0 \\
0 & 0 & 0 & 0 & -z_5 & y_6 \\
\end{pmatrix}\,.
\end{equation*}
The determinant of a $p\times p$ matrix of this type is $y_1\cdots y_p - z_1\cdots z_p$.  Therefore, we can write the determinant of $N$ as the product of powers of entries of $A$, powers of $\det A^{i,j}$, and binomials of this form.

\medskip
In our illustrative example, we simplify the determinant of $N_\textup{simplified}$ by expanding in the first row (contributing a factor of $a_{11}$), and we perform row and column operations to find the reduced matrix of $N$ to be
\begin{equation*}
N_\textup{reduced}=\begin{pmatrix} 
  a_{21} &  0 & -a_{22} \\
 -a_{31} &  a_{32} &  0 \\
  0 & -a_{22} &  a_{22} \\
\end{pmatrix},
\end{equation*}
whose determinant is $a_{21}a_{32}a_{22} - a_{31}a_{22}a_{22}$.

\medskip
Returning to the proof, the entries $y_q$ and $z_q$ are the variables $a_{ik}$, depending on in which vertex sets the rows lie and in which edge sets the columns lie.  If the vertices of $K_n$ corresponding to the rows in $B$ are labeled $v_1$ through $v_p$, this block of the block matrix corresponds to traversing the closed walk $C=v_1v_2\hdots v_pv_1$ in $K_n$ (in this direction).  As a result of the form of the simplified matrix of $N$, for a column that corresponds to an edge $e_q=v_qv_{q+1}$ in $E_k$ traversed from the vertex $v_q$ in vertex set $V_i$ to the vertex $v_{q+1}$ in vertex set $V_j$, the entry $y_q$ is $a_{ik}$ and the entry $z_q$ is $a_{jk}$.  (See Figure \ref{fig:cycle1}.)  Therefore each block $B$ in the block diagonal matrix contributes
\begin{equation}
\det B = \prod_{\substack{ e=v_qv_{q+1}\in C \\ e\in E_k, v_q\in V_i }} a_{ik} 
\ - \prod_{\substack{ e=v_qv_{q+1}\in C \\ e\in E_k, v_{q+1}\in V_j }} a_{jk}
\label{eq:BlockFactor}
\end{equation}
for some closed walk $C$ in $G$, whose length is $p$. 

\medskip
In our illustrative example, $N_\textup{reduced}$ is the incidence matrix of the closed walk  
$$v_3\stackrel{e_4}{\rightarrow}v_2\stackrel{e_1}{\rightarrow}v_1\stackrel{e_2}{\rightarrow}v_3,$$ where vertex $v_3$ is from $V_2$, vertex $v_2$ is from $V_3$, and vertex $v_3$ is from $V_2$.  Moreover, edge $e_4$ is from $E_1$, and edges $e_1$ and $e_2$ are from $E_2$.   Because we are working with a concrete example, we have not relabeled the vertices as we did in the preceding paragraph.

\begin{figure}[t]
\begin{center}
\begin{picture}(0,0)%
\includegraphics{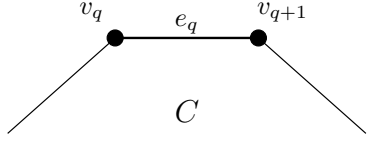}%
\end{picture}%
\setlength{\unitlength}{3947sp}%
\begingroup\makeatletter\ifx\SetFigFont\undefined%
\gdef\SetFigFont#1#2#3#4#5{%
  \reset@font\fontsize{#1}{#2pt}%
  \fontfamily{#3}\fontseries{#4}\fontshape{#5}%
  \selectfont}%
\fi\endgroup%
\begin{picture}(2274,894)(1264,-973)
\put(2401,-886){\makebox(0,0)[b]{\smash{\SetFigFont{12}{14.4}{\rmdefault}{\mddefault}{\updefault}{\color[rgb]{0,0,0}$C$}%
}}}
\put(2401,-286){\makebox(0,0)[b]{\smash{\SetFigFont{10}{12.0}{\familydefault}{\mddefault}{\updefault}{\color[rgb]{0,0,0}$e_q$}%
}}}
\put(3001,-211){\makebox(0,0)[b]{\smash{\SetFigFont{10}{12.0}{\familydefault}{\mddefault}{\updefault}{\color[rgb]{0,0,0}$v_{q+1}$}%
}}}
\put(1801,-211){\makebox(0,0)[b]{\smash{\SetFigFont{10}{12.0}{\rmdefault}{\mddefault}{\updefault}{\color[rgb]{0,0,0}$v_q$}%
}}}
\end{picture}
\end{center}
\caption{An edge $e_q=v_qv_{q+1}$ in the cycle $C$ generated by block $B$.  When $v_q\in V_i$, $v_{q+1}\in V_j$, and $e_q\in E_k$, the contributions $y_q$ and $z_q$ to $\det B$ are $a_{ik}$ and $a_{jk}$, respectively.} 
\label{fig:cycle1}
\end{figure}

\medskip 
Returning to the proof, we can simplify this expression by analyzing exactly what the $a_{ik}$ and $a_{jk}$ are.  Suppose that two consecutive edges $e_{q-1}$ and $e_{q}$ in $C$ are in the same edge set $E_k$, and suppose that the vertex $v_{q}$ that these edges share is in $V_i$. (See Figure \ref{fig:cycle2}.) In this case, both entries $z_{q-1}$ and $y_q$ are $a_{ik}$, which can then be factored out of each product in Equation~\eqref{eq:BlockFactor}.  

A particular case to mention is when the cycle $C$ contains a vertex that has multiple copies in $N$ (not necessarily both in $C$).  In this case, the edges of $C$ incident with this repeated vertex are both from the same edge set, as mentioned earlier.  After factoring out a multiplier for each pair of adjacent edges in the same edge set, all that remains inside the products in Equation~\eqref{eq:BlockFactor} is the contributions of multipliers from vertices where the incident edges are from different edge sets.

\begin{figure}[t]
\begin{center}
\begin{picture}(0,0)%
\includegraphics{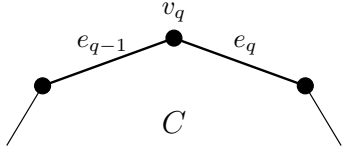}%
\end{picture}%
\setlength{\unitlength}{3947sp}%
\begingroup\makeatletter\ifx\SetFigFont\undefined%
\gdef\SetFigFont#1#2#3#4#5{%
  \reset@font\fontsize{#1}{#2pt}%
  \fontfamily{#3}\fontseries{#4}\fontshape{#5}%
  \selectfont}%
\fi\endgroup%
\begin{picture}(2124,969)(1339,-973)
\put(2851,-341){\makebox(0,0)[b]{\smash{\SetFigFont{10}{12.0}{\familydefault}{\mddefault}{\updefault}{\color[rgb]{0,0,0}$e_q$}%
}}}
\put(2401,-886){\makebox(0,0)[b]{\smash{\SetFigFont{12}{14.4}{\rmdefault}{\mddefault}{\updefault}{\color[rgb]{0,0,0}$C$}%
}}}
\put(2401,-136){\makebox(0,0)[b]{\smash{\SetFigFont{10}{12.0}{\familydefault}{\mddefault}{\updefault}{\color[rgb]{0,0,0}$v_{q}$}%
}}}
\put(1951,-341){\makebox(0,0)[b]{\smash{\SetFigFont{10}{12.0}{\familydefault}{\mddefault}{\updefault}{\color[rgb]{0,0,0}$e_{q-1}$}%
}}}
\end{picture}
\end{center}
\caption{Two consecutive edges $e_{q-1}$ and $e_q$, both incident with vertex $v_q$ in the cycle $C$ generated by block $B$.  When both edges are members of the same edge set $E_k$ and $v_q$ is a member of $V_i$, the contributions $z_{q-1}$ and $y_q$ are both $a_{ik}$, allowing this multiplier to be factored out of Equation \eqref{eq:BlockFactor}.}
\label{fig:cycle2}
\end{figure}

More precisely, when following the closed walk, let $I$ be the multiset of indices $i$ such that the walk $C$ passes from an edge in $E_2$ to an edge in $E_1$ at a vertex in $V_i$.  Similarly, let $J$ be the multiset of indices $j$ such that $C$ passes from an edge in $E_1$ to an edge in $E_2$ at a vertex in $V_j$.  Then what remains inside the products in Equation \eqref{eq:BlockFactor} after factoring out common multipliers is exactly 
\begin{equation*}
\det A^{I,J}= \prod_{i\in I} a_{i1} \prod_{j\in J} a_{j2} - \prod_{j\in J} a_{j2}\prod_{i\in I} a_{i2}.
\end{equation*}

There is one final simplifying step.  Consider a value $i$ occurring in both $I$ and $J$.  In this case, we can factor $a_{i1}a_{i2}$ out of both terms.  This implies that the determinant of each block $B$ of the block diagonal matrix is of the form
\begin{equation}
\pm \bigg(\prod_{i,k} a_{ik}^{s_{ik}} \bigg)
\det A^{I,J} ,
\label{eq:BlockFactor3}
\end{equation}
where the exponents $s_{i,k}$ are non-negative integers, $I$ and $J$ are disjoint subsets of $[m]$ of the same cardinality, and $2|I|+\sum_{i,k} s_{ik}=p$ because the degree of $\det B$ is the order of $B$.  Notice that when $|I|=|J|=1$ (say $I=\{i\}$ and $J=\{j\}$), the factor $\det A^{I,J}$ equals $\det A^{i,j}$.  Combining contributions from the simplification and reduction processes and from all blocks, we have the formula
\begin{equation}
\det N = \pm \prod_{i,j} (\det A^{i,j})^{|V_i\cap V_j|} \prod_{i,k} a_{ik}^{S_{ik}}
\prod_{B} \det A^{I_B,J_B},
\label{eq:detN}
\end{equation}
for some non-negative exponents $S_{ik}$.
We note that $$\sum 2|V_i \cap V_j| + \sum S_{ik} + \sum |I_B| = l.$$

\medskip
In the cycle in our illustrative example, vertex $v_3$ (in $V_2$) transitions from an edge in $E_2$ to an edge in $E_1$ and vertex $v_2$ (in $V_3$) transitions from an edge in $E_1$ to an edge in $E_2$. This implies that $I=\{2\}$ and $J=\{3\}$.  Vertex $v_1$ originally occurred in the two vertex sets $V_1$ and $V_2$; this implies that we can factor out the corresponding multiplier, $a_{22}$.  Indeed, the determinant of $N_\textup{reduced}$ is $a_{21}a_{32}a_{22}-a_{22}^2a_{31}=a_{22}\det A^{2,3}$.  Through these calculations we see that 
$$
\det N=\bigg(\frac{\det A^{1,2}}{a_{11}a_{22}}\bigg)a_{11}a_{22}\det A^{2,3}=\det A^{1,2}\det A^{2,3}.
$$

\subsection{The subdeterminant divides the formula}
We now verify that the product in Equation~\eqref{eq:detN} divides the right-hand side of Equation \eqref{eq:lcmr}.  The exponents $S_{ik}$ can be no larger than $n$ because there are only $n$ rows with entries $a_{ik}$ in $A\otimes D(K_n)$, so the expansion of the determinant, as a polynomial in the variable $a_{ik}$, has degree at most $n$.  Furthermore, it is not possible for the exponent of $a_{ik}$ to be $n$.  The only way this might occur is if $N$ were to contain in $V_i$ all $n$ vertices of $G$ and at least $n$ edges of $E_k$ incident with the vertices of $V_i$.  The corresponding set of columns is a dependent set of columns in $N$ (because the rank of $D(K_n)$ is $n-1$), which would make $\det N=0$. Therefore, $\det N$ contributes no more than $n-1$ factors of any $a_{ik}$ to any term of $\lcmd(A\otimes D(K_n))$. 

Now let us examine the exponents of factors of the form $\det A^{I,J}$ that may divide $\det N$.  Such factors may arise either upon the conversion of $N$ to the simplified matrix of $N$ if $|I|=|J|=1$, or from a block of the reduced matrix as in Equation \eqref{eq:BlockFactor3} if $|I|=|J|\geq1$.

The factors that arise in simplification come from duplicate pairs of vertices: every duplicated vertex $v$ leads to a factor $\det A^{i,j}$ where $v \in V_i \cap V_j$ (this is apparent in Equation \eqref{eq:detN}).  The total number of factors $\det A^{i,j}$ arising from simplification is $\sum_{\{i,j\}} |V_i \cap V_j| = d$, the number of duplicated vertices, which is not more than $n-1$ since $2d \leq l \leq 2(n-1)$.  Since each such factor divides $\lcmd A$, their product divides $(\lcmd A)^{n-1}$, the first component of Equation \eqref{eq:lcmr}.  

The factors $\det A^{I_B,J_B}$ from blocks $B$ of the reduced matrix arise from simple vertices---those which are not duplicated among the rows of $N$.  $I_B$ and $J_B$ are multisets of indices of vertex sets $V_i$ containing simple vertices, $I_B \cap J_B = \emptyset$, and $\sum_B (|I_B|+|J_B|) \leq c$, the number of simple vertices, since a simple vertex appears in only one block.  As $c \leq n$, $\sum_B (|I_B|+|J_B|) \leq n$.  Thus, the product of the corresponding determinants $\det A^{I_B,J_B}$ is one product in the $\LCM_\mathcal{K}$ component of Equation \eqref{eq:lcmr}.

\subsection{The formula is best possible}
We have shown that for every matrix $N$, $\det N$ divides the right-hand side of Equation \eqref{eq:lcmr}.  We now show that there exist graphs that attain the claimed powers of factors.  Consider the path of length $n-1$, $P=v_1v_2\cdots v_n$, as a subgraph of $K_n$.  Create the $(2n-2)\times(2n-2)$ submatrix $N$ of $A\otimes D(K_n)$ with rows corresponding to both an $i^\textup{th}$ copy and a $j^\textup{th}$ copy of vertices $v_1$ through $v_{n-1}$ and columns corresponding to two copies of every edge in $P$.  Then
\begin{equation*}
N = 
\left(\begin{array}{cccc:cccc} 
 a_{i1} &    0   &    0    &   0    &  a_{i2} &    0   &    0    &   0    \\
-a_{i1} & a_{i1} &    0    &   0    & -a_{i2} & a_{i2} &    0    &   0    \\
   0    & \ddots & \ddots  &   0    &    0    & \ddots & \ddots  &   0    \\
   0    &    0   & -a_{i1} & a_{i1} &    0    &    0   & -a_{i2} & a_{i2} \\ \hdashline
 a_{j1} &    0   &    0    &   0    &  a_{j2} &    0   &    0    &   0    \\
-a_{j1} & a_{j1} &    0    &   0    & -a_{j2} & a_{j2} &    0    &   0    \\
   0    & \ddots & \ddots  &   0    &    0    & \ddots & \ddots  &   0    \\
   0    &    0   & -a_{j1} & a_{j1} &    0    &    0   & -a_{j2} & a_{j2} \\ \end{array}\right),
\end{equation*}
with determinant $(\det A^{i,j})^{n-1}$.  The four quadrants of $N$ are $(n-1)\times(n-1)$ submatrices of $A\otimes D(K_n)$ with determinants $a_{i1}^{n-1}$, $a_{i2}^{n-1}$, $a_{j1}^{n-1}$, and $a_{j2}^{n-1}$, respectively.

We show that, for every collection $\mathcal{K}=\{(I_s,J_s)\}) \subseteq \mathcal{K}_m$ satisfying $2\sum|I_s|\leq n$, there is a submatrix $N$ of $A\otimes D(K_n)$ with determinant $\prod_{(I_s,J_s)\in\mathcal{K}}\det A^{I_s,J_s}$.  For each $s$, starting with $s=1$, choose $W_s \subseteq V(K_n)$ to consist of the lowest-numbered unused $n_s=2|I_s|$ vertices.  Thus, $W_s = \{v_{2k+1},\ldots,v_{2k+2n_s}\}$.  Take edges $v_{i-1}v_i$ for $2k+1<i\leq2k+2n_s$ and $v_{2k+1}v_{2k+2n_s}$.  This creates a cycle $C_s$ if $|I_s|>1$ and an edge $e_s$ if $|I_s|=1$.  For a cycle $C_s$, place each odd-indexed vertex of $W_s$ into a vertex set $V_i$ for every $i\in I_s$ and each even-indexed vertex into a vertex set $V_j$ for every $j\in J_s$.  For an edge $e_s$ corresponding to $I_s=\{i\}$ and $J_s=\{j\}$, place both vertices of $W_s$ in $V_i$ and $V_j$.  Place all edges of the form $v_{2m-1}v_{2m}$ into $E_1$ and all edges of the form $v_{2m}v_{2m+1}$ and $v_{2k+1}v_{2k+2n_s}$ into $E_2$.  Note that this puts $e_s$ into both $E_1$ and $E_2$.  

The submatrix $N$ of $A\otimes D(K_n)$ that arises from placing the vertices in numerical order and the edges in cyclic order along $C_s$ is the block-diagonal matrix where each block $N_s$ is a $2|I_s|\times 2|I_s|$ matrix of the form
\begin{equation*}
\begin{pmatrix}
 a_{i_11} &    0    &    0    & \cdots  &    0    & a_{i_12} \\ 
-a_{j_11} & -a_{j_12} &    0    & \cdots  &    0    &   0    \\
   0    &  a_{i_22} &  a_{i_21} &         &         &   0    \\
   0    &    0    & -a_{j_21} & -a_{j_22} &         & \vdots \\ 
 \vdots &  \vdots &         & \ddots  & \ddots  &   0    \\
   0    &    0    & \cdots  &         & -a_{j_l1} & -a_{j_l2}\\
\end{pmatrix}  
\end{equation*}
if $C_s$ is a cycle and 
$$
\begin{pmatrix} a_{i1} & a_{i2} \\ a_{j1} & a_{j2} \end{pmatrix}
$$
if $e_s$ is an edge.
The determinant of $N_s$ is exactly $\det A^{I_s,J_s}$, so the determinant of $N$ is 
$\prod_{(I_s,J_s)\in\mathcal{K}}\det A^{I_s,J_s}$, as desired.
\hfill\qedsymbol

\section{Acknowledgements}

The authors would like to thank an anonymous referee whose suggestions greatly improved the readability of our article.


\end{document}